\documentclass [12pt]{amsart}
\usepackage{amsmath}
\usepackage{amssymb}

\newtheorem{thm}{Theorem}[section]

\newtheorem{lem}[thm]{Lemma}

\theoremstyle{definition}

\theoremstyle{remark}
\newtheorem{rem}[thm]{Remark}
\numberwithin{equation}{section}

\newcommand{\beas}{\begin{eqnarray*}}
\newcommand{\eeas}{\end{eqnarray*}}
\newcommand{\bes} {\begin{equation*}}
\newcommand{\ees} {\end{equation*}}
\newcommand{\be} {\begin{equation}}
\newcommand{\ee} {\end{equation}}
\newcommand{\bea} {\begin{eqnarray}}
\newcommand{\eea} {\end{eqnarray}}
\newcommand{\ra} {\rightarrow}
\newcommand{\txt} {\textmd}
\newcommand{\ds} {\displaystyle}

\begin{document}

\title[\tiny{Segal-Bargmann transform and Paley-Wiener theorems on $M(2)$}]
 {Segal-Bargmann transform and Paley-Wiener theorems on $M(2)$}

\author{\tiny{E. K. Narayanan and Suparna Sen}}

\address{Department of Mathematics, Indian Institute of Science, Bangalore - 560012, India.}

\email{naru@math.iisc.ernet.in, suparna@math.iisc.ernet.in.}

\thanks{The first author was supported in part by a grant from UGC via DSA-SAP and the second author was supported by Shyama Prasad Mukherjee Fellowship from Council of Scientific and Industrial Research, India.}


\begin{abstract}
We study the Segal-Bargmann transform on $M(2).$ The range of this transform is characterized as a weighted Bergman space. In a similar fashion Poisson integrals are investigated. Using a Gutzmer's type formula we characterize the range as a class of functions extending holomorphically to an appropriate domain in the complexification of $M(2).$ We also prove a Paley-Wiener theorem for the inverse Fourier transform.

\vspace*{0.1in}

\begin{flushleft}
\vspace*{0.1in}
\end{flushleft}

\end{abstract}

\maketitle

\section{Introduction}

\noindent
Consider the following results from Euclidean Fourier analysis :

\indent
(I):A function $f \in L^2(\mathbb{R}^n)$ admits a factorization $f(x) = g*p_t(x)$ where $g \in L^2(\mathbb{R}^n)$ and $\ds{p_t(x) = \frac{1}{(4\pi t)^{\frac{n}{2}}} e^{\frac{-|x|^2}{4t}}}$ (the heat kernel on $\mathbb{R}^n$) if and only if $f$ extends as an entire function to $\mathbb{C}^n$ and  we have \\
$\ds{\frac{1}{(2\pi t)^{n/2}} \int_{\mathbb{C}^n} |f(z)|^2 e^{-\frac{|y|^2}{2t}} dx dy < \infty}$ $(z = x + iy).$ In this case we also have  $$ \| g\| _2 ^2 = \frac{1}{(2\pi t)^{n/2}} \int_{\mathbb{C}^n} |f(z)|^2 e^{-\frac{|y|^2}{2t}} dx dy.$$

The mapping $\ds{g \rightarrow g * p_t}$ is called the Segal-Bargmann transform and the above says that the Segal-Bargmann transform is a unitary map from $L^2(\mathbb{R}^n)$ onto $\ds{\mathcal{O}(\mathbb{C}^n) \bigcap L^2(\mathbb{C}^n, \mu)},$ where $\ds{d\mu(z) = \frac{1}{(2\pi t)^{n/2}} e^{-\frac{|y|^2}{2t}} dx dy }$ and $\mathcal{O}(\mathbb{C}^n)$ denotes the space of entire functions on $\mathbb{C}^n.$

\indent
(II) A function $f \in L^2(\mathbb{R})$ admits a holomorphic extension to the strip $\{ x + iy : |y| < t \}$ such that $$\ds{ \sup_{|y| \leq s} \int_\mathbb{R} |f(x+iy)|^2 dx < \infty ~ \forall s<t}$$ if and only if $$e^{s|\xi|} \tilde{f}(\xi) \in L^2(\mathbb{R}) ~ \forall ~  s < t$$ where $\tilde{f}$ denotes the Fourier transform of $f.$

\indent
(III) An $f \in L^2(\mathbb{R}^n)$ admits an entire extension to $\mathbb{C}^n$ such that $$|f(z)| \leq C_N (1 + |z|)^{-N} e^{R|Im z|} ~ \forall ~  z \in \mathbb{C}^n $$ if and only if $\tilde{f} \in \mathcal{C}_c^\infty(\mathbb{R}^n)$ and supp $\tilde{f} \subseteq \mathbb{B}(0,R), $ where $\mathbb{B}(0,R)$ is the ball of radius $R$ centered around $0$ in $\mathbb{R}^n.$

In this paper we aim to prove similar results for the non-commutative group $M(2) = \mathbb{R}^2 \ltimes SO(2).$ Some remarks are in order.

As noted above the map $\ds{g \rightarrow g * p_t}$ in (I) is called the Segal-Bargmann transform. This transform has attracted a lot of attention in the recent years mainly due to the work of Hall \cite{H} where a similar result was established for an arbitrary compact Lie group $K.$ Let $q_t$ be the heat kernel on $K$ and let $K_\mathbb{C}$ be the complexification of $K.$ Then Hall's result, Theorem 2 in \cite{H}, states that the map $f \rightarrow f * q_t$ is a unitary map  from $L^2(K)$ onto the Hilbert space of $\nu$-square integrable holomorphic functions on $K_\mathbb{C}$ for an appropriate positive $K$-invariant measure $\nu$ on $K_\mathbb{C}.$ Soon after Hall's paper a similar result was proved for compact symmetric spaces by Stenzel in \cite{S}. We also refer to \cite{HL}, \cite{HM}, \cite{KOS} and \cite{KTX} for similar results for other groups and spaces.

The second result above (II) is originally due to Paley and Wiener. Let \beas \mathcal{H}_t = \{ f \in L^2(\mathbb{R}), ~ f \txt{ has a holomorphic extension to } |Im z| < t \txt{ and } \\ \sup_{|y| \leq s} \int_\mathbb{R} |f(x+iy)|^2 dx < \infty ~ \forall ~  s < t \}. \eeas Then $\ds{\bigcup_{t>0} \mathcal{H}_t}$ may be viewed as the space of all analytic vectors for the regular representation of $\mathbb{R}$ on $L^2(\mathbb{R}).$ This point of view was further developed by R. Goodman (see \cite{G1} and \cite{G2}) who studied analytic vectors for representations of Lie groups. The theorem of Paley-Wiener (II) characterizes analytic vectors for the regular representation of $\mathbb{R}$ via a condition on the Fourier transform.

The third result (III) is the classical Paley-Wiener theorem. For a long time the Paley-Wiener theorem has been looked at as a characterization of the image (under Fourier transform) of $\mathcal{C}_c^\infty$ functions on the space you are interested in. Recently, using Gutzmer's formula Thangavelu \cite{T} has proved a Paley-Wiener type result for the inverse Fourier transform. See \cite{P} also for a similar result.

The plan of this paper is as follows : In the remaining of this section we recall the representation theory and Plancherel theorem of $M(2)$ and we prove the unitarity of the Segal-Bargmann transform. In the next section we study generalized Segal-Bargmann transform and prove an analogue of Theorem 8 and Theorem 10 in \cite{H}. Third section is devoted to a study of Poisson integrals on $M(2).$ This section is modeled after the work of Goodman \cite{G1} and \cite{G2}. In the final section we establish a Paley-Wiener type result for the inverse Fourier transform on $M(2).$

The rigid motion group $M(2)$ is the semi-direct product of $\mathbb{R}^2$ with $SO(2)$ (which will be identified with the circle group $S^1$) with the group law $$ (x,e^{i\alpha}) \cdot (y, e^{i\beta}) = (x + e^{i\alpha}y, e^{i(\alpha + \beta)}) \textmd{ where } x, y \in \mathbb{R}^2; e^{i\alpha}, e^{i\beta} \in S^1.$$

This group may be identified with a matrix subgroup of $GL(2,\mathbb{C})$ via the map $$ (x, e^{i\alpha}) \rightarrow \left ( \begin{matrix} e^{i \alpha} & x \\ 0 & 1 \\ \end{matrix} \right).$$

Unitary irreducible representations of $M(2)$ are completely described by Mackey's theory of induced representations. For any $\xi \in \mathbb{R}^2$ and $g \in M(2),$ we define $U_g^\xi$ as follows : $$ U_g^\xi F(s) = e^{i \langle x , s \xi \rangle} F(r^{-1}s) $$ for $g = (x,r)$ and $F \in L^2(S^1).$

It is known that $U_\xi$ is equivalent to $U_{\xi'}$ iff $|\xi| = |\xi'|.$ The above collection gives all the unitary irreducible representations of $M(2)$ sufficient for the Plancherel theorem to be true. The Plancherel theorem (see Theorem 4.2 of \cite{Su}) reads $$ \int_{M(2)} |f(g)|^2 dg = \int_{\mathbb{R}^2} \| \hat{f} (\xi) \|^2_{HS} d\xi $$ where $\hat{f} (\xi)$ is the 'group Fourier transform' defined as an operator from $L^2(S^1)$ to $L^2(S^1)$ by $$ \hat{f} (\xi) = \int_{M(2)} f(g) U_g^\xi dg.$$

Moreover, the group Fourier Transform $\hat{f}(\xi),$ for $\xi \in \mathbb{R}^2$ of $f \in L^1(M(2))$ is an integral operator with the kernel $k_f(\xi, e^{i\alpha}, e^{i\beta})$ where $$k_f(\xi, e^{i\alpha}, e^{i\beta}) = \tilde{f}(e^{i\beta}\xi, e^{i(\beta - \alpha)}),$$ and $\tilde{f}$ is the Euclidean Fourier Transform of $f$ in the $\mathbb{R}^2$-variable.

The Lie algebra of $M(2)$ is given by $ \left \{ \left ( \begin{matrix} i \alpha & x \\ 0 & 0 \\ \end{matrix} \right) : (x , e^{i\alpha}) \in M(2) \right \} .$ Let $$ \ds{X_1 = \left ( \begin{matrix} i & 0 \\ 0 & 0 \\ \end{matrix} \right) , ~~ X_2 = \left ( \begin{matrix} 0 & 1 \\ 0 & 0 \\ \end{matrix} \right) , ~~ X_3 = \left ( \begin{matrix} 0 & i \\ 0 & 0 \\ \end{matrix} \right) }$$

Then it is easy to see that $\ds{\{ X_1, X_2, X_3 \}}$ forms a basis for the Lie algebra of $M(2).$ The "Laplacian" $\ds{\Delta_{M(2)} = \Delta}$ is defined by $$\ds{\Delta = -(X_1^2 + X_2^2 + X_3^2).}$$ A simple computation shows that $\ds{\Delta = - \Delta_{\mathbb{R}^2} - \frac{\partial^2}{\partial \alpha^2}}$ where $\ds{\Delta_{\mathbb{R}^2}}$ is the Laplacian on $\mathbb{R}^2$ given by $\ds{\Delta_{\mathbb{R}^2} = \frac{\partial^2}{\partial x^2} + \frac{\partial^2}{\partial y^2}.}$ Since $\ds{\Delta_{\mathbb{R}^2}}$ and $\ds{\frac{\partial^2}{\partial \alpha^2}}$ commute, it follows that the heat kernel $\psi_t$ associated to $\Delta_{M(2)}$ is given by the product of the heat kernels $p_t$ on $\mathbb{R}^2$ and $q_t$ on $SO(2).$ In other words, $$ \psi_t (x, e^{i\alpha}) = p_t(x) q_t(e^{i\alpha}) = \frac{1}{4\pi t} e^{\frac{-|x|^2}{4t}} \sum_{n \in \mathbb{Z}} e^{-n^2 t} e^{in\alpha}.$$

Let $f \in L^2(M(2)).$ Expanding $f$ in the $SO(2)$ variable we obtain $$ f(x,e^{i\alpha}) = \sum_{m \in \mathbb{Z}} f_m(x) e^{im\alpha}$$ where $\ds{f_m(x) = \frac{1}{2 \pi} \int_{- \pi}^{\pi} f(x,e^{i \alpha}) e^{-im\alpha}} d\alpha$ and the convergence is understood in the $L^2$-sense. Since $p_t$ is radial (as a function on $\mathbb{R}^2$) a simple computation shows that $$ f * p_t (x, e^{i\alpha}) = \sum_{m \in \mathbb{Z}} f_m * p_t (x) e^{-m^2 t} e^{im\alpha}. $$

Let $\mathbb{C}^* = \mathbb{C} \backslash \{0\}$ and $\mathcal{H}(\mathbb{C}^2 \times \mathbb{C}^*)$ be the Hilbert space of holomorphic functions on $\mathbb{C}^2 \times \mathbb{C}^*$ which are square integrable with respect to $\mu \bigotimes \nu (z,w)$ where $$ d\mu(z) = \frac{1}{2 \pi t} e^{-\frac{|y|^2}{2t}} dx dy \textmd{ on } \mathbb{C}^2 $$ and $$ d\nu(w) = \frac{1}{2\pi} \frac{1}{\sqrt{2\pi t}} \frac{e^{-\frac{(\ln |w|)^2}{2t}}}{|w|^2} dw \textmd{ on } \mathbb{C}^*.$$ Using the Segal-Bargmann result for $\mathbb{R}^2$ and $S^1$ we can easily prove the following theorem:

\begin{thm}
If $f \in L^2(M(2)),$ then $f*p_t$ extends holomorphically to $\mathbb{C}^2 \times \mathbb{C}^*.$ Moreover, the map $f \rightarrow f * p_t$ is a unitary map from $L^2(M(2))$ onto $\mathcal{H}(\mathbb{C}^2 \times \mathbb{C}^*).$
\end{thm}

\section{Generalizations of Segal-Bargmann transform}

In this section we study generalizations of the Segal-Bargmann transform and prove an analogue of Theorem 8 and Theorem 10 in \cite{H}.

Let $\mu$ be any radial real-valued function on $\mathbb{R}^2$ such that
\begin{itemize}
\item $\mu(x) > 0 ~~ \forall ~  ~ x \in \mathbb{R}^2$ and is locally bounded away from zero,
\item $ \ds{ \forall ~  x \in \mathbb{R}^2, \sigma(x) = \int_{\mathbb{R}^2} e^{2x.y} \mu(y) dy < \infty .}$
\end{itemize} Define, for $z \in \mathbb{C}^2$ $$ \psi(z) = \int_{\mathbb{R}^2} \frac{e^{ia(y)}}{\sqrt{\sigma(y)}} e^{-iy.z} dy,$$ where $a$ is a real valued measurable function on $\mathbb{R}^2.$ Next, let $\nu$ be a measure on $\mathbb{C}^*$ such that
\begin{itemize}
\item $\nu$ is $S^1$-invariant,
\item $\nu$ is given by a positive density which is locally bounded away from zero,
\item $ \ds{ \forall ~  n \in \mathbb{Z}, \delta(n) = \int_{\mathbb{C}^*} |w|^{2n} d\nu(w) < \infty.}$
\end{itemize}
Define $ \ds {\chi (w) = \sum _ {n \in \mathbb{Z}} \frac{c_n}{\sqrt{\delta(n)}} w^n } $ for $w \in \mathbb{C}^*$ and $c_n \in \mathbb{C}$ such that $|c_n|=1.$ Also define $\ds{ \phi(z,w) = \psi(z) \chi(w) }$ for $z \in \mathbb{C}^2, ~ w \in \mathbb{C}^*.$ It is easy to see that $\phi(z,w)$ is a holomorphic function on $\mathbb{C}^2 \times \mathbb{C}^*.$ We have the following Paley-Wiener type theorem.

\begin{thm}
The mapping $$\ds{C_\phi f(z,w) = \int_{M(2)} f(\xi, e^{i\alpha}) \phi((\xi,e^{i\alpha})^{-1}(z,w))d\xi d\alpha}$$ is an isometric isomorphism of $L^2(M(2))$ onto $$\ds{\mathcal{O}(\mathbb{C}^2 \times \mathbb{C}^*) \bigcap L^2(\mathbb{C}^2 \times \mathbb{C}^*, \mu(y)dxdyd\nu(w)).}$$
\end{thm}

\begin{proof}
Let $f \in L^2(M(2))$ and \bea \label{eqn2} f(x,e^{i\alpha}) = \sum_{m \in \mathbb{Z}} f_m(x) e^{im\alpha} \eea where $\ds{f_m(x) = \frac{1}{2 \pi} \int_{- \pi}^{\pi} f(x,e^{i \alpha}) e^{-im\alpha}} d\alpha.$ Since the function $\phi(x, e^{i\alpha}),$ for $(x, e^{i\alpha}) \in M(2)$ is radial in the $\mathbb{R}^2$ variable $x,$ a simple computation shows that the Fourier series of $f * \phi(x, e^{i\alpha})$ is given by $$ f * \phi (x, e^{i\alpha}) = \sum_{m \in \mathbb{Z}} f_m * \psi (x) \frac{c_m}{\sqrt{\delta(m)}} e^{im\alpha}. $$

Now, notice that $f_m \in L^2(\mathbb{R}^2), ~ \forall ~  ~ m \in \mathbb{Z}$ and $f_m * \psi$ is a holomorphic function on $\mathbb{C}^2.$ Moreover, by Theorem 8 of \cite{H} we have \bea \label{eqn3} \int_{\mathbb{C}^2} |f_m * \psi(z)|^2 \mu(y) dx dy = \int_{\mathbb{R}^2} |f_m(x)|^2 dx. \eea

Naturally, the analytic continuation of $f * \phi(x, e^{i\alpha})$ to $\mathbb{C}^2 \times \mathbb{C}^*$ is given by \bea \label{eqn4} f * \phi (z,w) =  \sum_{m \in \mathbb{Z}} f_m * \psi (z) \frac{c_m}{\sqrt{\delta(m)}} w^m. \eea

We show that the series in (\ref{eqn4}) converges uniformly on compact sets in $\mathbb{C}^2 \times \mathbb{C}^*$ proving the holomorphicity. Let $K$ be a compact set in $\mathbb{C}^2 \times \mathbb{C}^*.$ For $(z,w) \in K$ we have \bea \label{eqn5} \left |\sum_{m \in \mathbb{Z}} f_m * \psi (z) \frac{c_m}{\sqrt{\delta(m)}} w^m \right| \leq \sum_{m \in \mathbb{Z}} |f_m * \psi (z)| \frac{|w|^m}{\sqrt{\delta(m)}} .\eea By Fourier inversion (see also Theorem 8 in \cite{H}) $$ f_m * \psi (z) = \int_{\mathbb{R}^2} \widetilde{f_m} (\xi) \frac{e^{ia(\xi)}}{\sqrt{\sigma(\xi)}} e^{-i\xi(x + iy)} d\xi $$ where $z = x + iy \in \mathbb{C}^2$ and $\widetilde{f_m}$ is the Fourier transform of $f_m.$ Hence, if $z$ varies in a compact subset of $\mathbb{C}^2,$ we have
\beas
|f_m * \psi(z)| &\leq& \|f_m\|_2 \left (\int_{\mathbb{R}^2} \frac{e^{2 \xi \cdot y}}{\sigma(\xi)} d\xi \right)^{\frac{1}{2}} \\
&\leq& C\|f_m\|_2.
\eeas

Using the above in (\ref{eqn5}) and assuming $|w| \leq R$ (as $w$ varies in a compact set in $\mathbb{C}^*$) we have $$ \left|\sum_{m \in \mathbb{Z}} f_m * \psi (z) \frac{c_m}{\sqrt{\delta(m)}} w^m \right| \leq C \sum_{m \in \mathbb{Z}} \|f_m\|_2 \frac{R^{m}}{\sqrt{\delta(m)}}.$$ Applying Cauchy-Schwarz inequality to the above, noting that $$\sum_{m \in \mathbb{Z}} \|f_m\|_2^2 = \|f\|_2^2 \txt{ and } \sum_{m \in \mathbb{Z}} \frac{R^{2m}}{\delta(m)} < \infty$$ we prove the above claim. Applying Theorem 10 in \cite{H} for $S^1$ we obtain $$ \int_{\mathbb{C}^*} |f * \phi(z,w)|^2 d\nu(w) = \sum_{m \in \mathbb{Z}} |f_m * \psi (z)|^2.$$

Integrating the above against $\mu(y)dx dy$ on $\mathbb{C}^2$ and using (\ref{eqn3}) we obtain that $$ \int_{\mathbb{C}^2} \int_{\mathbb{C}^*} |f * \phi(z,w)|^2 \mu(y) dx dy d\nu(w) = \|f\|_2^2.$$

To prove that the map $C_\phi$ is surjective it suffices to prove that the range of $C_\phi$ is dense in $\ds{\mathcal{O}(\mathbb{C}^2 \times \mathbb{C}^*) \bigcap L^2(\mathbb{C}^2 \times \mathbb{C}^*, \mu(y)dxdyd\nu(w)).}$ For this, consider functions of the form $f(x,e^{i\alpha}) = g(x) e^{im\alpha} \in L^2(M(2))$ where $g \in L^2(\mathbb{R}^2).$ Then a simple computation shows that $$ f * \phi(z,w) = g * \psi(z) w^m \txt{ for } (z,w) \in \mathbb{C}^2 \times \mathbb{C}^*.$$ Suppose $\ds{F \in \mathcal{O}(\mathbb{C}^2 \times \mathbb{C}^*) \bigcap L^2(\mathbb{C}^2 \times \mathbb{C}^*, \mu(y)dxdyd\nu(w))}$ be such that \bea \label{eqn6} \int_{\mathbb{C}^2 \times \mathbb{C}^*} F(z,w) \overline{g*\psi(z)} \overline{w}^m \mu(y)dxdyd\nu(w) = 0 \eea $\forall ~  g \in L^2(\mathbb{R}^2)$ and $\forall ~  m \in \mathbb{Z}.$ From (\ref{eqn6}) we have $$ \int_{\mathbb{C}^*} \left(\int_{\mathbb{C}^2} F(z,w) \overline{g*\psi(z)} d\mu(z) \right) \overline{w}^m d\nu(w) = 0, $$ which by Theorem 10 of \cite{H} implies that $$ \int_{\mathbb{C}^2} F(z,w) \overline{g*\psi(z)} d\mu(z) = 0 . $$ Finally, an application of Theorem 8 of \cite{H} shows that $F \equiv 0.$ Hence the proof.

\end{proof}

\section{Poisson integrals and Paley-Wiener type theorems}

In this section we study the Poisson integrals on $M(2).$ We also find conditions on the 'group Fourier transform' of a function so that it extends holomorphically to an appropriate domain in the complexification of the group. We start with the following Gutzmer-type lemma:

\begin{lem}\label{lemma}
Let $f \in L^2(M(2))$ extend holomorphically to the domain $$\{ (z,w) \in \mathbb{C}^2 \times \mathbb{C}^* : |Im z| < t, \frac{1}{R} < |w| < R \}$$  and $$\ds{\sup_{\left \{|y|<s , \frac{1}{N} < |w| < N \right\}} \int_{M(2)} |f(x+iy,|w|e^{i\theta})|^2 dx d\theta} < \infty $$ $\forall ~  s<t$ and $N<R.$ Then $$\ds{ \int_{M(2)} |f(x+iy,|w|e^{i\theta})|^2 dx d\theta = \sum_{n \in \mathbb{Z}} \left( \int_{\mathbb{R}^2} |\widetilde{f_n}(\xi)|^2 e ^{-2 \xi.y} d\xi \right) |w|^{2n} }$$ provided $|y|<t$ and $\ds{\frac{1}{R} < |w| < R}.$ Conversely, if  $$\ds{\sup_{\left\{|y|<s , \frac{1}{N} < |w| < N \right\}} \sum_{n \in \mathbb{Z}} \left( \int_{\mathbb{R}^2} |\widetilde{f_n}(\xi)|^2 e ^{-2 \xi.y} d\xi \right) |w|^{2n} < \infty} ~ \forall ~  s<t  \txt{ and } N<R$$ then $f$ extends holomorphically to the domain $$\ds{\left\{ (z,w) \in \mathbb{C}^2 \times \mathbb{C}^* : |Im z| < t, \frac{1}{R} < |w| < R \right\}}$$ and $$\ds{\sup_{\left\{|y|<s , \frac{1}{N} < |w| < N \right\}} \int_{M(2)} |f(x+iy,|w|e^{i\theta})|^2 dx d\theta} < \infty ~ \forall ~  s<t \txt{ and } N<R.$$

\end{lem}

\begin{proof}

Notice that $\ds{f_n(x) = \frac{1}{2\pi} \int_{-\pi}^{\pi} f(x,e^{i\alpha}) e^{-in\alpha} d\alpha.}$ It follows that $f_n(x)$ has a holomorphic extension to $\{ z \in \mathbb{C}^2 : |Im z| < t \}$ and $$\ds{ \sup_{|y|<s} \int_{\mathbb{R}^2} |f_n(x+iy)|^2 dx < \infty ~ \forall ~  s<t. }$$ Consequently, $$ \int_{\mathbb{R}^2} |f_n(x+iy)|^2 dx = \int_{\mathbb{R}^2} |\widetilde{f_n}(\xi)|^2 e ^{-2 \xi.y} d\xi \txt{ for } |y|<s ~ \forall ~  s<t. $$ Now, for each fixed $z \in \mathbb{C}^2$ with $|Im z|<s$ the function $w \ra f(z,w)$ is holomorphic in the annulus $\{ w \in \mathbb{C}^* : \frac{1}{N} < |w| < N \}$ for every $s<t$ and $N<R$ and so admits a Laurent series expansion $$ f(z,w) = \sum _{m \in \mathbb{Z}} a_m(z) w^m. $$ It follows that $a_m(z) = f_m(z) ~ \forall ~  m \in \mathbb{Z}.$ First part of the lemma is proved now by appealing to the Plancherel theorem on $S^1$ and $\mathbb{R}^2.$ Converse can also be proved similarly.

\end{proof}

Recall from the introduction that the Laplacian $\Delta$ on $M(2)$ is given by $\ds{\Delta = -\Delta_{\mathbb{R}^2}}$ $\ds{- \frac{\partial^2}{\partial \alpha^2}.}$ If $f \in L^2(M(2))$ it is easy to see that $$ e^{-t\Delta^{\frac{1}{2}}} f (x,e^{i \alpha}) = \sum_{m \in \mathbb{Z}} \left( \int_{\mathbb{R}^2} \widetilde{f_m}(\xi) e^{-t(|\xi|^2 + m^2)^{\frac{1}{2}}} e^{ix.\xi} d\xi \right) e^{im\alpha}.$$ We have the following (almost) characterization of the Poisson integrals. Let $\Omega_s$ denote the domain in $\mathbb{C}^2 \times \mathbb{C}^*$ defined by $$\ds{ \Omega_s = \{ (z,w) : |Im z| < s, e^{-s} < |w| < e^s \}. }$$

\begin{thm}\label{theorem}

Let $f \in L^2(M(2)).$ Then $g = e^{-t\Delta^{\frac{1}{2}}} f$ extends to a holomorphic function on the domain $\Omega_{\frac{t}{\sqrt{2}}}$ and $$ \sup_{\left\{|y|<\frac{t}{\sqrt{2}} , e^{-\frac{t}{\sqrt{2}}} < |w| < e^{\frac{t}{\sqrt{2}}}\right\}} \int_{M(2)} |g(x+iy,|w|e^{i\alpha})|^2 dx d\alpha < \infty .$$ Conversely, let $g$ be a holomorphic function on $\Omega_t$ and $$ \ds{\sup_{\left\{|y|<s , e^{-s} < |w| < e^s \right\}} \int_{M(2)} |g(x+iy,|w|e^{i\alpha})|^2 dx d\alpha < \infty \txt{ for } s<t.}$$ Then $\forall ~  s<t, ~ \exists f \in L^2(M(2))$ such that $e^{-s\Delta^{\frac{1}{2}}} f = g.$

\end{thm}

\begin{proof}

We know that, if $f \in L^2(M(2))$ then $$ g (x,e^{i \alpha}) = e^{-t\Delta^{\frac{1}{2}}} f (x,e^{i \alpha}) = \sum_{m \in \mathbb{Z}} \left( \int_{\mathbb{R}^2} \widetilde{f_m}(\xi) e^{-t(|\xi|^2 + m^2)^{\frac{1}{2}}} e^{ix.\xi} d\xi \right) e^{im\alpha} .$$ Also, $\ds{g(x,e^{i\alpha}) = \sum_{m \in \mathbb{Z}}g_m(x) e^{im\alpha} }$ with $\ds{ \widetilde{g_m}(\xi) = \widetilde{f_m}(\xi) e^{-t(|\xi|^2 + m^2)^{\frac{1}{2}}}}.$ If $s \leq \frac{t}{\sqrt{2}}$ it is easy to show that $$ \sup_{\left\{\xi \in \mathbb{R}^2, m \in \mathbb{Z}\right\}} e^{-2t(|\xi|^2 + m^2)^{\frac{1}{2}}} e^{2|\xi| |y|} e^{2|m|s} \leq C < \infty \txt{ for } |y| \leq \frac{t}{\sqrt{2}}. $$ It follows that $$ \sup_{\left\{|y|<\frac{t}{\sqrt{2}} , e^{-\frac{t}{\sqrt{2}}} < |w| < e^{\frac{t}{\sqrt{2}}}\right\}} \sum_{m \in \mathbb{Z}} \left( \int_{\mathbb{R}^2} |\widetilde{g_m}(\xi)|^2 e ^{-2 \xi.y} d\xi \right) |w|^{2m} < \infty. $$ By the previous lemma we prove the first part of the theorem.

Conversely, let $g$ be a holomorphic function on $\Omega_t$ and $$\ds{\sup_{\{|y|<s , e^{-s} < |w| < e^s\}} \int_{M(2)} |g(x+iy,|w|e^{i\alpha})|^2 dx d\alpha < \infty \txt{ for } s<t.}$$ By Lemma \ref{lemma} we have $$\ds{\sup_{\{|y|<s , e^{-s} < |w| < e^s\}} \sum_{n \in \mathbb{Z}} \left( \int_{\mathbb{R}^2} |\widetilde{g_n}(\xi)|^2 e ^{-2 \xi.y} d\xi \right) |w|^{2n} < \infty \txt{ for } s<t.}$$ Integrating the above over $|y|=s<t,$ we obtain $$ \sum_{n \in \mathbb{Z}} \left( \int_{\mathbb{R}^2} |\widetilde{g_n}(\xi)|^2 J_0(i2s|\xi|) d\xi \right) |w|^{2n} < \infty,$$ where $J_0$ is the Bessel function of first kind. Noting that $J_0(i2s|\xi|) \sim e^{2s|\xi|}$ for large $|\xi|$ we obtain $$ \sum_{n \in \mathbb{Z}} \int_{\mathbb{R}^2} |\widetilde{g_n}(\xi)|^2 e^{2s|\xi|} e^{2|n|s} d\xi < \infty \txt{ for } s<t. $$ This surely implies that $$ \sum_{n \in \mathbb{Z}} \int_{\mathbb{R}^2} |\widetilde{g_n}(\xi)|^2 e^{2s(|\xi|^2 + m^2)^{\frac{1}{2}}} d\xi < \infty \txt{ for } s<t. $$ Defining $\widetilde{f_m}(\xi)$ by $\widetilde{f_m}(\xi) = \widetilde{g_m}(\xi) e^{s(|\xi|^2 + m^2)^{\frac{1}{2}}}$ we obtain $$f(x,e^{i\alpha}) = \sum_{m \in \mathbb{Z}} f_m(x) e^{im\alpha} \in L^2(M(2))$$ and $g = e^{-s\Delta^{\frac{1}{2}}} f.$

\end{proof}

\begin{rem}
A similar result may be proved for the operator $\ds{e^{-t\Delta^{\frac{1}{2}}_{\mathbb{R}^2}} e^{-t\Delta^{\frac{1}{2}}_{S^1}}.}$
\end{rem}

\subsection*{Analytic vectors}

Let $\pi$ be a unitary representation of a Lie group $G$ on a Hilbert space $H.$ A vector $v \in H$ is called an analytic vector for $\pi$ if the function $g \rightarrow \pi(g)v$ is analytic.

Recall the representations $U_g^{\xi}$ from the introduction. Denote by $U_g^a$ the representations $U_g^{(a,0)}$ for $a>0.$ If $e_n(\theta) = e^{in\theta} \in L^2(S^1),$ it is easy to see that $e_n$'s are analytic vectors for these representations. For $g = (x,e^{i\alpha}) \in M(2),$ we have $$ \left( U_g^a e_n \right) (\theta) = e^{i\langle x,ae^{i\theta} \rangle} e^{in(\theta - \alpha)}.$$

This action of $U_g^a$ on $e_n$ can clearly be analytically continued to $\mathbb{C}^2 \times \mathbb{C}^*$ and we obtain $$ \left( U_{(z,w)}^a e_n \right) (\theta) = e^{i\langle x,ae^{i\theta} \rangle} e^{-\langle y,ae^{i\theta} \rangle} w^{-n} e^{in\theta} $$ where $(z,w) \in \mathbb{C}^2 \times \mathbb{C}^*$ and $z=x+iy \in \mathbb{C}^2.$

We also note that the action of $S^1$ on $\mathbb{R}^2$ naturally extends to an action of $\mathbb{C}^*$ on $\mathbb{C}^2$ given by $$ w(z_1,z_2) = (z_1 \cos \zeta - z_2 \sin \zeta, z_1 \sin \zeta + z_2 \cos \zeta)$$ where $w=e^{i\zeta} \in \mathbb{C}^*$ and $(z_1,z_2) \in \mathbb{C}^2.$

Our next theorem is in the same spirit of Theorem 3.1 from Goodman \cite{G2}.

\begin{thm}
Let $f \in L^2(M(2)).$ Then $f$ extends holomorphically to $\mathbb{C}^2 \times \mathbb{C}^*$ with $$\int_{|y| = r} \int_{M(2)} |f(w^{-1}(x+iy),|w|e^{i\alpha})|^2 dx d\alpha d\sigma_r(y) < \infty$$ (where $\sigma_r$ is the normalized surface area measure on the sphere $\{ |y| = r \} \subseteq \mathbb{R}^2$) iff $$\ds{ \int_0^\infty \int_{|y| = r} \|U_{(z,w)}^a \hat{f}(a) \|_{HS}^2 d\sigma_r(y) ada} < \infty $$ where $z=x+iy \in \mathbb{C}^2$ and $w \in \mathbb{C}^*.$ In this case we also have \beas && \int_0^\infty \int_{|y| = r} \|U_{(z,w)}^a \hat{f}(a) \|_{HS}^2 d\sigma_r(y) ada \\ &=& \int_{|y| = r} \int_{M(2)} |f(w^{-1}(x+iy),|w|e^{i\alpha})|^2 dx d\alpha d\sigma_r(y).\eeas
\end{thm}

\begin{proof}

First assume that $f \in L^2(M(2))$ satisfies the transformation property \bea \label{e1} f(e^{i\theta}x,e^{i\alpha}) = e^{im\theta} f(x,e^{i\alpha}) \eea for some fixed $m \in \mathbb{Z}$ and $\forall ~ (x,e^{i\alpha}) \in M(2).$ As earlier we have $$ \left( \hat{f}(a) e_n \right) (\theta) = \widetilde{f_n}(ae^{i\theta}) e^{in\theta}.$$ By the Hecke-Bochner identity, we have $$ \widetilde{f_n}(ae^{i\theta}) = i^{-|m|} a^{|m|} (\mathcal{F}_{2+2|m|} g)(a) e^{im\theta}$$ where $\mathcal{F}_{2+2|m|} (g)$ is the $2+2|m|$-dimensional Fourier transform of $g(x) =$ $\ds{\frac{f_n(|x|)}{|x|^{|m|}}},$ considered as a radial function on $\mathbb{R}^{2+2|m|}.$

Hence, $$ \left( \hat{f}(a) e_n \right) (\theta) = i^{-|m|} a^{|m|}(\mathcal{F}_{2+2|m|} g)(a) e^{i(m+n)\theta}.$$ It follows that $\hat{f}(a) e_n$ is an analytic vector and we can apply $U_{(z,w)}^a$ to the above. We obtain \beas && \left( U_{(z,w)}^a \hat{f}(a) e_n \right) (\theta) \\ &=& e^{i\langle x,ae^{i\theta} \rangle} e^{-\langle y,ae^{i\theta} \rangle} i^{-|m|} a^{|m|} (\mathcal{F}_{2+2|m|} g)(a) w^{-(m+n)} e^{i(m+n)\theta}.\eeas Thus, \beas && \int_{S^1} \left| \left[ U_{(z,w)}^a \hat{f}(a) e_n \right] (\theta)\right|^2 d\theta\\ &=& |w|^{-2(m+n)} \int_{S^1} a^{2m} |(\mathcal{F}_{2+2m} g)(a)|^2 e^{-2\langle y,ae^{i\theta} \rangle} d\theta \\ &=& |w|^{-2(m+n)} \int_{S^1} e^{-2\langle y,ae^{i\theta} \rangle} |\widetilde{f_n}(ae^{i\theta})|^2 d\theta. \eeas Hence, \bea \label{e2} \left. \begin{array}{rcll} && \ds{\int_0^\infty \| U_{(z,w)}^a \hat{f}(a) \|_{HS}^2 ada} \\ &=& \ds{|w|^{-2m} \sum_{n \in \mathbb{Z}} |w|^{-2n} \left( \int_{\mathbb{R}^2} e^{-2\langle y,\xi \rangle} |\widetilde{f_n}(\xi)|^2 d\xi \right)}. \end{array}\right. \eea Notice that, if $f$ extends holomorphically to $\mathbb{C}^2 \times \mathbb{C}^*$ we must have $f(w^{-1}z,w)=w^{-m}f(z,w) ~ \forall ~ (z,w) \in \mathbb{C}^2 \times \mathbb{C}^*$ because of (\ref{e1}).

In view of Lemma \ref{lemma}, the above remark and the identity (\ref{e2}), the theorem is established for functions with transformation property (\ref{e1}) and we obtain \bea \label{e3} \left. \begin{array}{rcll} && \ds{ \int_0^\infty \int_{|y| = r} \|U_{(z,w)}^a \hat{f}(a) \|_{HS}^2 d\sigma_r(y) ada} \\ &=& \ds{ \int_{|y| = r} \int_{M(2)} |f(w^{-1}(x+iy),|w|e^{i\alpha})|^2 dx d\alpha d\sigma_r(y).} \end{array}\right. \eea

Next, we deal with the general case. For $f \in L^2(M(2))$ define $$ f^m(x,e^{i\alpha}) = \int_{S^1} f(e^{i\theta}x,e^{i\alpha}) e^{-im\theta} d\theta.$$ Then $f^m(e^{i\theta}x,e^{i\alpha}) = e^{im\theta} f^m(x,e^{i\alpha})$ and $f_m$'s are orthogonal on $M(2).$ Assume that $f$ extends holomorphically to $\mathbb{C}^2 \times \mathbb{C}^*.$ Then, so does $f^m ~ \forall ~ m \in \mathbb{Z}$ and we have \bea \label{e4} \left. \begin{array}{rcll} && \ds{\int_{|y| = r} \int_{M(2)} |f(w^{-1}(x+iy),|w|e^{i\alpha})|^2 dx d\alpha d\sigma_r(y)} \\ &=& \ds{\sum_{m \in \mathbb{Z}} \int_{|y| = r} \int_{M(2)} |f^m(w^{-1}(x+iy),|w|e^{i\alpha})|^2 dx d\alpha d\sigma_r(y).} \end{array}\right. \eea This follows from the fact that $$ \int_{|y|=r} \int_{\mathbb{R}^2} f^m(w^{-1}(x+iy),w) \overline{f^l(w^{-1}(x+iy),w)} dx d\sigma_r(y) = 0 \txt{ if } m \neq l.$$ Applying identity (\ref{e3}) we get from (\ref{e4}) \bea \label{e5} \sum_{m \in \mathbb{Z}} \int_0^\infty \int_{|y|=r} \| U_{(z,w)}^a \widehat{f^m}(a) \|_{HS}^2 ada d\sigma_r(y) < \infty. \eea Now, let $\ds{\langle T,S \rangle_{HS} = \sum_{n \in \mathbb{Z}} \langle Te_n, Se_n \rangle}$ denote the inner product in the space of Hilbert-Schmidt operators on $L^2(S^1).$ Then we notice that \bea \label{e6} \left. \begin{array}{rcll} && \ds{\int_{|y|=r} \left \langle U_{(z,w)}^a \widehat{f^m}(a) , U_{(z,w)}^a \widehat{f^l}(a) \right \rangle_{HS} d\sigma_r(y)} \\ &=& \ds{ \delta_{ml} \sum_{n \in \mathbb{Z}} (-1)^{-l} \left( \frac{a}{i}\right)^{m+l} (\mathcal{F}_{2+2|m|} g^m)(a)}\\ && \ds{\overline{(\mathcal{F}_{2+2|l|} g^l)(a)} w^{-(m+n)} (\overline{w})^{-(l+n)} J_0(2ira).} \end{array}\right. \eea Hence (\ref{e5}) implies that $$ \int_0^\infty \int_{|y|=r} \| U_{(z,w)}^a \widehat{f}(a) \|_{HS}^2 ada d\sigma_r(y) < \infty.$$

To prove the converse, we first show that $f$ has a holomorphic extension to whole of $\mathbb{C}^2 \times \mathbb{C}^*.$ Recall that we have $$ \left( \hat{f}(a) e_n \right) (\theta) = \widetilde{f_n}(ae^{i\theta}) e^{in\theta}.$$ Expanding $\widetilde{f_n}(ae^{i\theta})$ into Fourier series we have $$ \widetilde{f_n}(ae^{i\theta}) = \sum_{k \in \mathbb{Z}} C_{a,n}(k) e^{ik\theta}.$$ Hence $\left( U_{(z,w)}^a \hat{f}(a) e_n \right) (\theta)$ is given by $$ \sum_{k \in \mathbb{Z}} C_{a,n}(k) e^{i\langle x,ae^{i\theta} \rangle} e^{-\langle y,ae^{i\theta} \rangle} w^{-(k+n)} e^{i(k+n)\theta}.$$ Thus $$ \int_{|y|=r} \int_{S^1} \left| \left[ U_{(z,w)}^a \hat{f}(a) e_n \right] (\theta)\right|^2 d\theta d\sigma_r(y) = J_0(2ira) \sum_{k \in \mathbb{Z}} |C_{a,n}(k)|^2 |w|^{-(k+n)}.$$ Notice that $J_0(2ira) \sim e^{2ra}$ for large $a$ and $\ds{\sum_{k \in \mathbb{Z}} |C_{a,n}(k)|^2 = \int_{S^1} |\widetilde{f_n}(ae^{i\theta})|^2 d\theta}.$ If $e^{-r} < |w| < e^r,$ we obtain $$ \sum_{n \in \mathbb{Z}} \left( \int_0^\infty \int_{S^1} |\widetilde{f_n}(ae^{i\theta})|^2 e^{2ra} d\theta ada \right) |w|^{2n} < \infty,$$ which implies $$ \sum_{n \in \mathbb{Z}} \left( \int_{\mathbb{R}^2} |\widetilde{f_n}(\xi)|^2 e^{2r|\xi|} d\xi \right) |w|^{2n} < \infty.$$ Since this is true for all $r>0$ and $w \in \mathbb{C}^*,$ by Lemma \ref{lemma} $f$ extends holomorphically to $\mathbb{C}^2 \times \mathbb{C}^*.$ It follows that $f^m(x,e^{i\alpha})$ defined by $$ f^m(x,e^{i\alpha}) = \int_{S^1} f(e^{i\theta}x,e^{i\alpha}) e^{-im\theta} d\theta$$ also extends holomorphically to $\mathbb{C}^2 \times \mathbb{C}^*$ and $$\int_{M(2)} |f^m((x+iy),|w|e^{i\alpha})|^2 dx d\alpha < \infty.$$ Now the proof can be completed using the identity (\ref{e3}), orthogonality of $U_{(z,w)}^a \widehat{f^m}(a)$ (see (\ref{e6})) and (\ref{e4}).

\end{proof}

\section{A Paley-Wiener theorem for the inverse Fourier transform}

Recall from the introduction that the 'group Fourier Transform' $\hat{f}(\xi),$ for $\xi \in \mathbb{R}^2$ of $f \in L^1(M(2))$ is an integral operator with the kernel $k_f(\xi, e^{i\alpha}, e^{i\beta})$ where $k_f(\xi, e^{i\alpha}, e^{i\beta}) = \tilde{f}(e^{i\beta}\xi, e^{i(\beta - \alpha)}), ~ \tilde{f}$ being the Euclidean Fourier Transform of $f$ in the $\mathbb{R}^2$-variable. We have the following Paley-Wiener theorem for the inverse Fourier Transform :

\begin{thm}
Let $f \in L^1(M(2))$ be such that $\hat{f}(\xi) \equiv 0 ~~ \forall ~  ~ |\xi| > R$ and the kernel $k_f$ of $\hat{f}(\xi)$ is smooth on $\mathbb{R}^2 \times S^1 \times S^1.$ Then $x \rightarrow f(x,e^{i\alpha})$ extends to an entire function of exponential type $R$ such that \be \label{eqn1} \sup_{e^{i\alpha} \in S^1} |z^m f(z,e^{i\alpha})| \leq c_m e^{R|Im z|}, ~~ \forall ~  z \in \mathbb{C}^2, ~ \forall ~  m \in \mathbb{N}^2 .\ee Conversely, if $f$ extends to an entire function on $\mathbb{C}^2$ in the first variable and satisfies (\ref{eqn1}) then $\hat{f}(\xi) = 0$ for $|\xi|>R$ and $k_f$ is smooth on $\mathbb{R}^2$.
\end{thm}

\begin{proof}
We have $$ (\hat{f}(\xi)F)(e^{i\alpha}) = \int_{S^1} k_f(\xi, e^{i\alpha}, e^{i\beta}) F(e^{i\beta}) d\beta , ~~ \txt{ for } F \in L^2(S^1)$$ where $k_f(\xi, e^{i\alpha}, e^{i\beta}) = \tilde{f}(e^{i\beta}\xi, e^{i(\beta - \alpha)}).$

Assume that $\hat{f}(\xi) \equiv 0 ~~ \forall ~ |\xi| > R.$ Since $k_f$ is smooth we have $\tilde{f}(\cdot, e^{i\alpha}) \in \mathcal{C}_c^\infty(\mathbb{R}^2).$ By the Paley-Wiener theorem on $\mathbb{R}^2$ we obtain that $f(\cdot, e^{i\alpha})$ extends to an entire function on $\mathbb{C}^2$ of exponential type. Moreover, if $m \in \mathbb{N}^2$ $$ z^m f(z,e^{i\alpha}) = \int_{|\xi| \leq R} \frac{\partial^m \tilde{f} }{\partial \xi^m} (\xi, e^{i\alpha}) e^{i \xi \cdot z} d\xi .$$

It follows that $$ \sup_{e^{i\alpha} \in S^1} |z^m f(z, e^{i\alpha})| \leq \sup_{e^{i\alpha} \in S^1} \left \|  \frac{\partial^m \tilde{f} }{\partial \xi^m} (\cdot, e^{i\alpha}) \right \|_1 e^{R |Im z|} .$$

Conversely, if $f$ is holomorphic on $\mathbb{C}^2$ and satisfies (\ref{eqn1}), by the Paley-Wiener theorem on $\mathbb{R}^2$ we get $\tilde{f}(\cdot, e^{i\alpha}) \in \mathcal{C}_c^\infty(\mathbb{R}^2)$ and $\tilde{f}(\xi, e^{i\alpha}) = 0 ~ \forall ~ |\xi| > R.$ Hence $\hat{f}(\xi) \equiv 0 ~~ \forall ~ |\xi| > R.$ Moreover, $k_f$ is smooth on $\mathbb{R}^2$ since $\tilde{f}$ is smooth.
\end{proof}

\end{document}